\documentclass[12pt]{amsart}
\usepackage{amsmath, amsthm, amscd, amsfonts, amssymb, graphicx, color,float,pgf,tikz}
\usepackage{amssymb,fontenc}
\usepackage{latexsym,wasysym,mathrsfs}
\usepackage{hyperref}
\usetikzlibrary{arrows}
\usepackage[square,numbers,sort&compress]{natbib}

\makeatletter \oddsidemargin.9375in \evensidemargin \oddsidemargin
\newtheorem{theorem}{Theorem}[section]
\newtheorem{lemma}[theorem]{Lemma}
\newtheorem{proposition}[theorem]{Proposition}

\theoremstyle{definition}
\newtheorem{definition}[theorem]{Definition}

\numberwithin{equation}{section}

\newcommand{\be}{\begin{equation}}
\newcommand{\ee}{\end{equation}}

\numberwithin{equation}{section}

\usepackage{etoolbox}
\makeatletter
\patchcmd{\@settitle}{\uppercasenonmath\@title}{}{}{}
\patchcmd{\@setauthors}{\MakeUppercase}{}{}{}
\makeatother

\allowdisplaybreaks
\begin{document}

\setcounter{page}{1}
\title[Controlled continuous g-Frames in Hilbert $C^{\ast}$-Modules]{Controlled continuous g-Frames in Hilbert $C^{\ast}$-Modules}

\author[H. LABRIGUI$^{*}$, A. TOURI, S. KABBAJ]{H. LABRIGUI$^1$$^{\ast}$, A. TOURI$^1$,  \MakeLowercase{and} S. KABBAJ$^1$}

\address{$^{1}$Department of Mathematics, University of Ibn Tofail, B.P. 133, Kenitra, Morocco}
\email{\textcolor[rgb]{0.00,0.00,0.84}{  hlabrigui75@gmail;  touri.abdo68@gmail.com;samkabbaj@yahoo.fr}}


\keywords{Continuous g-Frames, Controlled continuous g-frames, $C^{\ast}$-algebra, Hilbert $\mathcal{A}$-modules.\\
\indent
\\
\indent $^{*}$ Corresponding author}
\maketitle
\begin{abstract}
Frame Theory has a great revolution in recent years, this Theory have been extended from Hilbert spaces to Hilbert  $C^{\ast}$-modules. The purpose of this paper is the introduction and the study of the concept of Controlled  Continuous g-Frames in Hilbert $C^{\ast}$-Modules. Also we give some properties.
\end{abstract}

\section{Introduction and preliminaries}
The concept of frames in Hilbert spaces has been introduced by Duffin and Schaeffer \cite{Duf} in 1952 to study some deep problems in nonharmonic Fourier
series, after the fundamental paper \cite{13} by Daubechies, Grossman and Meyer, frames
theory began to be widely used, particularly in the more specialized context of wavelet frames and Gabor frames \cite{Gab}.

Hilbert $C^{\ast}$-module arose as generalizations of the notion Hilbert space. The basic idea was to consider modules over $C^{\ast}$-algebras instead of linear spaces and to allow the inner product to take values in the $C^{\ast}$-algebras \cite{LEC}. 

Continuous frames defined by Ali, Antoine and Gazeau \cite{STAJP}. Gabardo and Han in \cite{14} called these kinds frames, frames associated with measurable spaces. For more details, the reader can refer to \cite{MR1}, \cite{MR2} and \cite{ARAN}.

Theory of frames have been extended from Hilbert spaces to Hilbert  $C^{\ast}$-modules \cite{Frank}, \cite{R1}, \cite{R2}, \cite{R3}, \cite{R4}.




In the following we briefly recall the definitions and basic properties of $C^{\ast}$-algebra, Hilbert $\mathcal{A}$-modules. Our reference for $C^{\ast}$-algebras is \cite{{Dav},{Con}}. For a $C^{\ast}$-algebra $\mathcal{A}$ if $a\in\mathcal{A}$ is positive we write $a\geq 0$ and $\mathcal{A}^{+}$ denotes the set of positive elements of $\mathcal{A}$.
\begin{definition}\cite{Con}.	
	Let $ \mathcal{A} $ be a unital $C^{\ast}$-algebra and $\mathcal{H}$ be a left $ \mathcal{A} $-module, such that the linear structures of $\mathcal{A}$ and $ \mathcal{H} $ are compatible. $\mathcal{H}$ is a pre-Hilbert $\mathcal{A}$-module if $\mathcal{H}$ is equipped with an $\mathcal{A}$-valued inner product $\langle.,.\rangle_{\mathcal{A}} :\mathcal{H}\times\mathcal{H}\rightarrow\mathcal{A}$, such that is sesquilinear, positive definite and respects the module action. In the other words,
	\begin{itemize}
		\item [(i)] $ \langle x,x\rangle_{\mathcal{A}}\geq0 $ for all $ x\in\mathcal{H} $ and $ \langle x,x\rangle_{\mathcal{A}}=0$ if and only if $x=0$.
		\item [(ii)] $\langle ax+y,z\rangle_{\mathcal{A}}=a\langle x,y\rangle_{\mathcal{A}}+\langle y,z\rangle_{\mathcal{A}}$ for all $a\in\mathcal{A}$ and $x,y,z\in\mathcal{H}$.
		\item[(iii)] $ \langle x,y\rangle_{\mathcal{A}}=\langle y,x\rangle_{\mathcal{A}}^{\ast} $ for all $x,y\in\mathcal{H}$.
	\end{itemize}	 
	For $x\in\mathcal{H}, $ we define $||x||=||\langle x,x\rangle_{\mathcal{A}}||^{\frac{1}{2}}$. If $\mathcal{H}$ is complete with $||.||$, it is called a Hilbert $\mathcal{A}$-module or a Hilbert $C^{\ast}$-module over $\mathcal{A}$. For every $a$ in $C^{\ast}$-algebra $\mathcal{A}$, we have $|a|=(a^{\ast}a)^{\frac{1}{2}}$ and the $\mathcal{A}$-valued norm on $\mathcal{H}$ is defined by $|x|=\langle x, x\rangle_{\mathcal{A}}^{\frac{1}{2}}$ for all $x\in\mathcal{H}$.
	
	Let $\mathcal{H}$ and $\mathcal{K}$ be two Hilbert $\mathcal{A}$-modules, A map $T:\mathcal{H}\rightarrow\mathcal{K}$ is said to be adjointable if there exists a map $T^{\ast}:\mathcal{K}\rightarrow\mathcal{H}$ such that $\langle Tx,y\rangle_{\mathcal{A}}=\langle x,T^{\ast}y\rangle_{\mathcal{A}}$ for all $x\in\mathcal{H}$ and $y\in\mathcal{K}$.
	
We reserve the notation $End_{\mathcal{A}}^{\ast}(\mathcal{H},\mathcal{K})$ for the set of all adjointable operators from $\mathcal{H}$ to $\mathcal{K}$ and $End_{\mathcal{A}}^{\ast}(\mathcal{H},\mathcal{H})$ is abbreviated to $End_{\mathcal{A}}^{\ast}(\mathcal{H})$.

\end{definition}

The following lemmas will be used to prove our mains results
\begin{lemma} \label{1} \cite{Pas}.
	Let $\mathcal{H}$ be Hilbert $\mathcal{A}$-module. If $T\in End_{\mathcal{A}}^{\ast}(\mathcal{H})$, then $$\langle Tx,Tx\rangle\leq\|T\|^{2}\langle x,x\rangle \qquad \forall x\in\mathcal{H}.$$
\end{lemma}

\begin{lemma} \label{sb} \cite{Ara}.
	Let $\mathcal{H}$ and $\mathcal{K}$ two Hilbert $\mathcal{A}$-modules and $T\in End^{\ast}(\mathcal{H},\mathcal{K})$. Then the following statements are equivalent:
	\begin{itemize}
		\item [(i)] $T$ is surjective.
		\item [(ii)] $T^{\ast}$ is bounded below with respect to norm, i.e., there is $m>0$ such that $\|T^{\ast}x\|\geq m\|x\|$ for all $x\in\mathcal{K}$.
		\item [(iii)] $T^{\ast}$ is bounded below with respect to the inner product, i.e., there is $m'>0$ such that $\langle T^{\ast}x,T^{\ast}x\rangle\geq m'\langle x,x\rangle$ for all $x\in\mathcal{K}$.
	\end{itemize}
\end{lemma}
\begin{lemma} \label{3} \cite{Deh}.
	Let $\mathcal{H}$ and $\mathcal{K}$ two Hilbert $\mathcal{A}$-modules and $T\in End^{\ast}(\mathcal{H},\mathcal{K})$. Then:
	\begin{itemize}
		\item [(i)] If $T$ is injective and $T$ has closed range, then the adjointable map $T^{\ast}T$ is invertible and $$\|(T^{\ast}T)^{-1}\|^{-1}\leq T^{\ast}T\leq\|T\|^{2}.$$
		\item  [(ii)]	If $T$ is surjective, then the adjointable map $TT^{\ast}$ is invertible and $$\|(TT^{\ast})^{-1}\|^{-1}\leq TT^{\ast}\leq\|T\|^{2}.$$
	\end{itemize}	
\end{lemma}

\section{Controlled continuous g-Frames in Hilbert $C^{\ast}$-Modules}
Let $X$ be a Banach space, $(\Omega,\mu)$ a measure space, and function $f:\Omega\to X$ a measurable function. Integral of the Banach-valued function $f$ has defined Bochner and others. Most properties of this integral are similar to those of the integral of real-valued functions. Because every $C^{\ast}$-algebra and Hilbert $C^{\ast}$-module is a Banach space thus we can use this integral and its properties.

Let $(\Omega,\mu)$ be a measure space, let $U$ and $V$ be two Hilbert $C^{\ast}$-modules, $\{V_{w}\}_{w\in\Omega}$  is a sequence of subspaces of V, and $End_{\mathcal{A}}^{\ast}(U,V_{w})$ is the collection of all adjointable $\mathcal{A}$-linear maps from $U$ into $V_{w}$.
We define
\begin{equation*}
	\oplus_{w\in\Omega}V_{w}=\left\{x=\{x_{w}\}_{w\in\Omega}: x_{w}\in V_{w}, \left\|\int_{\Omega}|x_{w}|^{2}d\mu(w)\right\|<\infty\right\}.
\end{equation*}
For any $x=\{x_{w}\}_{w\in\Omega}$ and $y=\{y_{w}\}_{w\in\Omega}$, if the $\mathcal{A}$-valued inner product is defined by $\langle x,y\rangle=\int_{\Omega}\langle x_{w},y_{w}\rangle d\mu(w)$, the norm is defined by $\|x\|=\|\langle x,x\rangle\|^{\frac{1}{2}}$, the $\oplus_{w\in\Omega}V_{w}$ is a Hilbert $C^{\ast}$-module.
Let $GL^{+}(U)$ be the set for all positive bounded linear invertible operators on $U$ with bounded inverse.
\begin{definition}\cite{MRKAN}
	We call $\{\Lambda_{w}\in End_{\mathcal{A}}^{\ast}(U,V_{w}): w\in\Omega\}$ a continuous g-frame for Hilbert $C^{\ast}$-module $U$ with respect to $\{V_{w}: w\in\Omega\}$ if:
	\begin{itemize}
		\item for any $x\in U$, the function $\tilde{x}:\Omega\rightarrow V_{w}$ defined by $\tilde{x}(w)=\Lambda_{w}x$ is measurable;
		\item there exist two strictly nonzero elements $A$ and $B$ in $\mathcal{A}$ such that
		\begin{equation} \label{2.1}
			A\langle x,x\rangle \leq\int_{\Omega}\langle\Lambda_{w}x,\Lambda_{w}x\rangle d\mu(w)\leq B\langle x,x\rangle , \forall x\in U.
		\end{equation}
	\end{itemize}
	The elements $A$ and $B$ are called continuous g-frame bounds. 
	
	If $A=B$ we call this continuous g-frame a continuous tight g-frame, and if $A=B=1_{\mathcal{A}}$ it is called a continuous Parseval g-frame. If only the right-hand inequality of \eqref{2.1} is satisfied, we call $\{\Lambda_{w}: w\in\Omega\}$ a 
	continuous g-Bessel sequence for $U$ with respect to $\{V_{w}: w\in\Omega\}$ with Bessel bound $B$.\\
	The contnuous g-frame operator $S$ on $U$ is :
	\begin{equation*}
	Sx=\int_{\Omega}\Lambda^{\ast}_{\omega}\Lambda_{\omega}xd\mu (\omega)
	\end{equation*}
	The frame operator S is a bounded, positive, selfadjoint, and invertible (see \cite{MRKAN})
\end{definition}

\begin{theorem}\cite{MRKAN}
Let $\Lambda= \{\Lambda_{w}\in End_{\mathcal{A}}^{\ast}(U,V_{w}): w\in\Omega\}$, then $\Lambda$ be a continuous g-frame for $U$ with respect to $\{V_{w}: w\in\Omega\}$ if and only if there exist a constants $A$ and $B$ such that for any $x\in U$ :
\begin{equation*}\label{ssp}
	A\|x\|^{2} \leq\| \int_{\Omega}\langle\Lambda_{w}x,\Lambda_{w}x\rangle d\mu(w)\| \leq B\|x\|^{2}
\end{equation*}
\end{theorem}
\begin{definition}
	Let $C,C^{'} \in GL^{+}(U)$, we call $\Lambda= \{\Lambda_{w}\in End_{\mathcal{A}}^{\ast}(U,V_{w}): w\in\Omega\}$ a $(C-C^{'})$-controlled continuous g-frame for Hilbert $C^{\ast}$-module $U$ with respect to  $\{V_{w}: w\in\Omega\}$ if $\Lambda$ is continuous g-Bessel sequence and there exist two constants $A > 0$ and $B<\infty$ such that :
	\begin{equation}
		A\langle x,x\rangle \leq\int_{\Omega}\langle\Lambda_{w}Cx,\Lambda_{w}C^{'}x\rangle d\mu(w)\leq B\langle x,x\rangle , \forall x\in U.
	\end{equation}
	$A$ and $B$ are called the $(C-C^{'})$-controlled continuous g-frames bounds.\\
	If $C^{'}=I$ then we call $\Lambda$ a $C$-controlled continuous g-frames for $U$ with respect to  $\{V_{w}: w\in\Omega\}$.\\
\end{definition}

Let $\Lambda= \{\Lambda_{w}\in End_{\mathcal{A}}^{\ast}(U,V_{w}): w\in\Omega\}$ be a continuous g-frames for $U$ with respect to $\{V_{w}: w\in\Omega\}$.\\
The bounded linear operator  $T_{CC^{'}}:l^{2}(\{V_{w}\}_{w\in\Omega}) \rightarrow U$ given by
\begin{equation*}
T_{CC^{'}}(\{y_{w}\}_{w\in\Omega})=\int_{\Omega}(CC^{'})^{\frac {1}{2}}\Lambda^{\ast}_{\omega}y_{\omega}d\mu(w) \qquad  \forall \{y_{w}\}_{w\in\Omega} \in l^{2}(\{V_{w}\}_{w\in\Omega})
\end{equation*}
is called the synthesis operator for the  $(C-C^{'})$-controlled continuous g-frame $ \{\Lambda_{w}\}_{ w\in\Omega}$.\\
The adjoint operator $T^{\ast}_{CC^{'}}: U\rightarrow l^{2}(\{V_{w}\}_{w\in\Omega})$ given by
\begin{equation}\label{2..3}
T^{\ast}_{CC^{'}}(x)=\{\Lambda_{\omega}(C^{'}C)^{\frac{1}{2}}x\}_{\omega \in \Omega} \qquad \forall x\in U
\end{equation}
is called the analysis operator for the  $(C-C^{'})$-controlled continuous g-frame $ \{\Lambda_{w} w\in\Omega\}$.\\
When $C$ and $C^{'}$ commute with each other, and commute with the operator $\Lambda^{\ast}_{\omega}\Lambda_{\omega}$ for each $\omega \in \Omega$, then the $(C-C^{'})$-controlled continuous g-frames operator: \\

$S_{CC^{'}}:U\longrightarrow U$ is defined as: 
$S_{CC^{'}}x=T_{CC^{'}}T^{\ast}_{CC^{'}}x=\int_{\Omega}C^{'}\Lambda^{\ast}_{w}\Lambda_{w}Cx d\mu(w)$\\
From now on we assume that $C$ and $C^{'}$ commute with each other, and commute with the operator $\Lambda^{\ast}_{\omega}\Lambda_{\omega}$ for each $\omega \in \Omega$
\begin{proposition}
The $(C-C^{'})$-controlled continuous g-frames operator $
S_{CC^{'}}$ is bounded, positive, sefladjoint and invertible.
\end{proposition}
\begin{proof}\textcolor{white}{.}
	We show that $S_{CC^{'}}$ is a bounded operator:
	\begin{equation*}
\|S_{CC^{'}}\|=\underset{x \in U, \|x\|\leq1}{\sup}\|\langle S_{CC^{'}}x,x\rangle \|=\underset{x \in U, \|x\|\leq1}{\sup}\int_{\Omega}C^{'}\Lambda^{\ast}_{w}\Lambda_{w}Cx d\mu(w)\|\leq B
	\end{equation*}
From the $(C-C^{'})$-controlled continuous g-frames identity (2.2), we have:
\begin{equation*}
	A\langle x,x\rangle \leq \langle S_{CC^{'}}x,x\rangle \leq B\langle x,x\rangle
\end{equation*} 
so
\begin{equation*}
A.Id_{U} \leq S_{CC^{'}} \leq B.Id_{U}
\end{equation*} 
Where $Id_{U}$ is the identity operator in $U$.\\
We clearly see that $S_{CC^{'}}$ is a positive operator.\\
Thus the $(C-C^{'})$-controlled continuous g-frames operator  $S_{CC^{'}}$is bounded and invertible\\
In other hand we know every positive operator is self adjoint.
\end{proof}
\begin{theorem}\label{spp}
Let $\Lambda= \{\Lambda_{w}\in End_{\mathcal{A}}^{\ast}(U,V_{w}): w\in\Omega\}$ is $(C-C^{'})$-controlled continuous g-bessel sequence for $U$, then $\Lambda$ is a $(C-C^{'})$-controlled continuous g-frames for $U$ with respect to $\{V_{w}: w\in\Omega\}$ if and only if there exist a positive constants $A$ and $B$ such that : 
\begin{equation} \label{2.3}
	A\|x\|^{2} \leq\| \int_{\Omega}\langle\Lambda_{w}Cx,\Lambda_{w}C^{'}x\rangle d\mu(w)\| \leq B\|x\|^{2}  \qquad  \forall x\in U.
\end{equation}
\end{theorem}
\begin{proof}
Let $\{\Lambda_{w}\}_{w\in\Omega}$ be a $(C-C^{'})$-controlled continuous g-frames for $U$ with respect to $\{V_{w}: w\in\Omega\}$ with bounds $A$ and $B$.\\
Hence, we have 
\begin{equation}\label {2.4}
A\langle x,x\rangle \leq \int_{\Omega}\langle\Lambda_{w}Cx,\Lambda_{w}C^{'}x\rangle d\mu(w) \leq B\langle x,x\rangle  \qquad \forall x\in U.
\end{equation}
Since $0 \leq \langle x,x\rangle$, $ \forall x\in U $ , then we can take the norme in the left, middle and right termes of the above inequality \eqref{2.4}.\\
Thus we have:
\begin{equation*}
\|A\langle x,x\rangle\| \leq \| \int_{\Omega}\langle\Lambda_{w}Cx,\Lambda_{w}C^{'}x\rangle d\mu(w)\| \leq \|B\langle x,x\rangle\|  \qquad \forall x\in U.
\end{equation*}
So,
\begin{equation*}
A\|x\|^{2} \leq \| \int_{\Omega}\langle\Lambda_{w}Cx,\Lambda_{w}C^{'}x\rangle d\mu(w)\| \leq B\|x\|^{2} \qquad \forall x\in U.
\end{equation*}
Conversely, suppose that \eqref{2.3} holds,we have:\\
\begin{equation}\label{2.5}
\langle S^{\frac{1}{2}}_{CC^{'}}x,S^{\frac{1}{2}}_{CC^{'}}x\rangle=\langle S_{CC^{'}}x,x\rangle=\int_{\Omega}\langle\Lambda_{w}Cx,\Lambda_{w}C^{'}x\rangle d\mu(w)
\end{equation}
using \eqref{2.5} in \eqref{2.4}, we get for all $x\in U$:
\begin{equation*}
\sqrt{A}\|x\| \leq \|S^{\frac{1}{2}}_{CC^{'}}x\|\leq \sqrt{B}\|x\| 
\end{equation*}
by lemma \ref{sb} $\exists m, M > 0 $ such that :
\begin{equation*}
 m\langle x,x\rangle \leq \langle S^{\frac{1}{2}}_{CC^{'}}x,S^{\frac{1}{2}}_{CC^{'}}x\rangle \leq M\langle x,x\rangle
\end{equation*}
Therefore $\{\Lambda_{w} : w\in\Omega\}$ is a $(C-C^{'})$-controlled continuous g-frames for $U$ with respect to  $\{V_{w}: w\in\Omega\}$ 
\end{proof}
\begin{theorem}
Let $C \in GL^{+}(U)$, the sequence $\Lambda= \{\Lambda_{w}\in End_{\mathcal{A}}^{\ast}(U,V_{w}): w\in\Omega\}$ is a continuous g-frame for $U$ with respect to $\{V_{w}: w\in\Omega\}$ if and only if $\Lambda$ is a  $(C-C)$-controlled continuous g-frames for $U$ with respect to $\{V_{w}: w\in\Omega\}$  
\end{theorem}
\begin{proof}
Suppose that $\{\Lambda_{w} : w\in\Omega\}$ is $(C-C)$-controlled continuous g-frames with bounds $A$ and $B$, then :
\begin{equation*}
A\|x\|^{2} \leq\| \int_{\Omega}\langle\Lambda_{w}Cx,\Lambda_{w}Cx\rangle d\mu(w)\| \leq B\|x\|^{2}  \qquad \forall x\in U.
\end{equation*}
For any $x\in U$, we have:
\begin{align*}
A\|x\|^{2}=A\|CC^{-1}x\|^{2}&\leq A\|C\|^{2}\|C^{-1}x\|^{2}\\
&\leq \|C\|^{2}\|\int_{\Omega}\langle\Lambda_{w}CC^{-1}x,\Lambda_{w}CC^{-1}x\rangle d\mu(w)\|\\
&=\|C\|^{2}\|\int_{\Omega}\langle\Lambda_{w}x,\Lambda_{w}x\rangle d\mu(w)\|\\
\end{align*}
hence,
\begin{equation}\label{a}
A\|C\|^{-2}\|x\|^{2}\leq \|\int_{\Omega}\langle\Lambda_{w}x,\Lambda_{w}x\rangle d\mu(w)\|
\end{equation}
on the other hand :
\begin{equation*}
\|\int_{\Omega}\langle\Lambda_{w}x,\Lambda_{w}x\rangle d\mu(w)\|=\|\int_{\Omega}\langle\Lambda_{w}CC^{-1}x,\Lambda_{w}CC^{-1}x\rangle d\mu(w)\|
\end{equation*}
\begin{equation}\label{b}
\|\int_{\Omega}\langle\Lambda_{w}x,\Lambda_{w}x\rangle d\mu(w)\|
\leq B\|C^{-1}x\|^{2}\leq B\|C^{-1}\|^{2}\|x\|^{2}
\end{equation}
From \eqref{a} and \eqref{b} and theorem\ref{ssp} we conclude that $\{\Lambda_{w}, w\in\Omega\}$ is a continuous g-frame with bounds $A\|C\|^{-2}$ and $B\|C^{-1}\|^{2}$\\
Conversely, let $\{\Lambda_{w}, w\in\Omega\}$ is a continuous g-frame with bounds $E$ and $F$, then for all $x\in U$ we have :\\
\begin{equation*}
E\langle x,x\rangle \leq\int_{\Omega}\langle\Lambda_{w}x,\Lambda_{w}x\rangle d\mu(w)\leq F\langle x,x\rangle 
\end{equation*}
So, for all $x\in U$, $Cx\in U$, and :
\begin{equation}\label{c}
\int_{\Omega}\langle\Lambda_{w}Cx,\Lambda_{w}Cx\rangle d\mu(w)\leq F\langle Cx,Cx\rangle \leq F\|C\|^{2}\langle x,x\rangle
\end{equation}
Also, for all $x\in U$,
\begin{equation*}
E\langle x,x\rangle=E\langle C^{-1}Cx,C^{-1}Cx\rangle\leq E\|C^{-1}\|^{2}\langle Cx,Cx\rangle
\end{equation*}
then,
\begin{equation}\label{cc}
E\langle x,x\rangle\leq \|C^{-1}\|^{2}\int_{\Omega}\langle\Lambda_{w}Cx,\Lambda_{w}Cx\rangle d\mu(w)
\end{equation}
From \eqref{c} and \eqref{cc}, we have:
\begin{equation*}
E\|C^{-1}\|^{-2}\langle x,x\rangle\leq \int_{\Omega}\langle\Lambda_{w}Cx,\Lambda_{w}Cx\rangle d\mu(w)\leq F\|C\|^{2}\langle x,x\rangle
\end{equation*}
Hence $\Lambda$ is a  $(C-C)$-controlled continuous g-frames with bounds  $E\|C^{-1}\|^{-2}$ and $F\|C\|^{2}$

\end{proof}
\begin{proposition}
Let $\{\Lambda_{w}, w\in\Omega\}$ is a continuous g-frame for $U$ with respect to $\{V_{w}: w\in\Omega\}$ and $S$ the continuous g-frame operator associated. Let $C,C^{'} \in GL^{+}(U)$, then $\{\Lambda_{w}, w\in\Omega\}$ is  $(C-C^{'})$-controlled continuous g-frames	
\end{proposition}
\begin{proof}
Let $\{\Lambda_{w}, w\in\Omega\}$ is a continuous g-frame with bounds $A$ and $B$.\\by theorem \eqref{ssp} we have:
\begin{equation*}
A\|x\|^{2} \leq\| \int_{\Omega}\langle\Lambda_{w}x,\Lambda_{w}x\rangle d\mu(w)\| \leq B\|x\|^{2} \qquad \forall x\in U
\end{equation*} 
\begin{equation}\label{2.10}
\Longrightarrow A\|x\|^{2} \leq\| \langle Sx,x\rangle \| \leq B\|x\|^{2} \qquad \forall x\in U
\end{equation} 
and 
\begin{equation}\label{2.11}
\|\int_{\Omega}\langle\Lambda_{w}Cx,\Lambda_{w}C^{'}x\rangle d\mu(w)\|=\|C\| \|C^{'}\|\| \int_{\Omega}\langle\Lambda_{w}x,\Lambda_{w}x\rangle d\mu(w)\|=\|C\| \|C^{'}\|\|\langle Sx,x\rangle \|
\end{equation}
From \eqref{2.10} and \eqref{2.11}, we have :
\begin{equation*}
A\|C\| \|C^{'}\|\|x\|^{2}\leq \|\int_{\Omega}\langle\Lambda_{w}Cx,\Lambda_{w}C^{'}x\rangle d\mu(w)\|\leq B\|C\| \|C^{'}\|\|x\|^{2} \qquad \forall x\in U
\end{equation*}
we conclude by theoreme \ref{spp} that $\{\Lambda_{w}, w\in\Omega\}$ is  $(C-C^{'})$-controlled continuous g-frames	with bounds $A\|C\| \|C^{'}\|$ and $B\|C\| \|C^{'}\|$
\end{proof}
\begin{theorem}\label{sspp}
Let $\{\Lambda_{w}, w\in\Omega\} \subset End^{*}_{A}(U,V_{\omega})$ and let $C,C^{'} \in GL^{+}(U)$ so that $C,C^{'}$ commute with each other and commute with $\Lambda^{\ast}_{\omega}\Lambda_{\omega}$ for all $\omega \in \Omega$. Then the following are equivalent :\\
(1) the sequence $\{\Lambda_{w}, w\in\Omega\}$ is a $(C-C^{'})$-controlled continuous g-Bessel sequence for $U$ with respect $\{V_{\omega}\}_{\omega \in \Omega}$ with bounds $A$ and $B$\\
(2) The operator $T_{CC^{'}}:l^{2}(\{V_{w}\}_{w\in\Omega}) \rightarrow U$ given by
\begin{equation*}
T_{CC^{'}}(\{y_{w}\}_{w\in\Omega})=\int_{w\in\Omega}(CC^{'})^{\frac {1}{2}}\Lambda^{\ast}_{\omega}y_{\omega}d\mu(w) \qquad \forall \{y_{w}\}_{w\in\Omega} \in l^{2}(\{V_{w}\}_{w\in\Omega})
\end{equation*}
is well defined and bounded operator with $\|T_{CC^{'}}\|\leq \sqrt{B}$
\end{theorem}
\begin{proof}
$(1)\Longrightarrow (2)$\\
Let  $\{\Lambda_{w}, w\in\Omega\}$ be a $(C-C^{'})$-controlled continuous g-Bessel sequence for $U$ with respect $\{V_{\omega}\}_{\omega \in \Omega}$ with bound $B$.\\
From theorem \ref{spp} we have :
\begin{equation} \label{2.12}
\|\int_{\Omega}\langle\Lambda_{w}Cx,\Lambda_{w}C^{'}x\rangle d\mu(w)\| \leq B\|x\|^{2}  \qquad \forall x\in U.
\end{equation}
For any sequence $\{y_{w}\}_{w\in \Omega} \in l^{2}(\{V_{\omega}\}_{\omega \in \Omega})$
\begin{align*}
\|T_{CC^{'}}(\{y_{w}\}_{w\in\Omega})\|^{2}&=\underset{x \in U, \|x\|=1}{\sup}\|\langle T_{CC^{'}}(\{y_{w}\}_{w\in\Omega}),x\rangle \|^{2}\\
&=\underset{x \in U, \|x\|=1}{\sup}\|\langle \int_{\Omega}(CC^{'})^{\frac {1}{2}}\Lambda^{\ast}_{\omega}y_{\omega}d\mu(w),x\rangle\|^{2}\\
&=\underset{x \in U, \|x\|=1}{\sup}\| \int_{\Omega}\langle(CC^{'})^{\frac {1}{2}}\Lambda^{\ast}_{\omega}y_{\omega},x\rangle d\mu(w)\|^{2}\\ 
&=\underset{x \in U, \|x\|=1}{\sup}\| \int_{\Omega}\langle y_{\omega},\Lambda_{\omega}(CC^{'})^{\frac {1}{2}}x\rangle d\mu(w)\|^{2}\\ 
&\leq \underset{x \in U, \|x\|=1}{\sup}\| \int_{\Omega}\langle y_{\omega},y_{\omega}\rangle d\mu(w)\| \|\int_{\Omega}\langle \Lambda_{\omega}(CC^{'})^{\frac {1}{2}}x,\Lambda_{\omega}(CC^{'})^{\frac {1}{2}}x\rangle d\mu(w)\|\\ 
&= \underset{x \in U, \|x\|=1}{\sup}\| \int_{\Omega}\langle y_{\omega},y_{\omega}\rangle d\mu(w)\| \|\int_{\Omega}\langle \Lambda_{\omega}Cx,\Lambda_{\omega}C^{'}x\rangle d\mu(w)\|\\ 
&\leq \underset{x \in U, \|x\|=1}{\sup}\| \int_{\Omega}\langle y_{\omega},y_{\omega}\rangle d\mu(w)\|B\|x\|^{2} =B\|\{y_{\omega}\}_{\omega \in \Omega}\|^{2}
\end{align*}
Then we have 
\begin{align*}
\|T_{CC^{'}}(\{y_{w}\}_{w\in\Omega})\|^{2}\leq B\|\{y_{\omega}\}_{\omega \in \Omega}\|^{2} \Longrightarrow \|T_{CC^{'}}\|\leq \sqrt{B} 
\end{align*}
we conclude the operator $T_{CC^{'}}$ is well defined and bounded\\
$(2)\Longrightarrow (1)$\\
Let the operator $T_{CC^{'}}$ is well defined, bounded and $\|T_{CC^{'}}\|\leq \sqrt{B}$ \\
For any $x\in U$ and finite subset $\Psi \subset \Omega$, we have:
\begin{align*}
\int_{\Psi}\langle\Lambda_{w}Cx,\Lambda_{w}C^{'}x\rangle d\mu(w)&=\int_{\Psi}\langle C^{'}\Lambda^{\ast}_{w}\Lambda_{w}Cx,x\rangle d\mu(w)\\
&=\int_{\Psi}\langle (CC^{'})^{\frac {1}{2}}\Lambda^{\ast}_{w}\Lambda_{w}(CC^{'})^{\frac {1}{2}}x,x\rangle d\mu(w)\\
&=\langle T_{CC^{'}}(\{y_{w}\}_{w\in\Psi}),x\rangle\\
&\leq \| T_{CC^{'}}\|\|(\{y_{w}\}_{w\in\Psi})\|\|x\|
\end{align*}
Where: $y_{w}=\Lambda_{w}(CC^{'})^{\frac {1}{2}}x$ if $\omega \in \Psi$ and $y_{w}=0$ if $\omega \notin \Psi$\\
Therefore, 
\begin{align*}
\int_{\Psi}\langle\Lambda_{w}Cx,\Lambda_{w}C^{'}x\rangle d\mu(w) &\leq \| T_{CC^{'}}\|(\int_{\Psi}\|\Lambda_{w}(CC^{'})^{\frac {1}{2}}x\|^{2}d\mu(w))^{\frac {1}{2}}\|x\|\\
&=\| T_{CC^{'}}\|(\int_{\Psi}\langle\Lambda_{w}Cx,\Lambda_{w}C^{'}x\rangle d\mu(w))^{\frac {1}{2}}\|x\|
\end{align*}
Since $\Psi$ is arbitrary, we have:
\begin{align*}
\int_{\Omega}\langle\Lambda_{w}Cx,\Lambda_{w}C^{'}x\rangle d\mu(w) &\leq \| T_{CC^{'}}\|^{2}\|x\|^{2}
\end{align*}
\begin{align*}
\Longrightarrow \int_{\Omega}\langle\Lambda_{w}Cx,\Lambda_{w}C^{'}x\rangle d\mu(w) &\leq B\|x\|^{2} \qquad as:  \| T_{CC^{'}}\|\leq \sqrt{B}
\end{align*}
Therfore $\{\Lambda_{w}, w\in\Omega\}$ is a $(C-C^{'})$-controlled continue g-Bessel sequence for $U$ with respect to $\{V_{\omega}\}_{\omega \in \Omega}$
\end{proof}
\begin{proposition}
Let  $\Lambda= \{\Lambda_{w}\in End_{\mathcal{A}}^{\ast}(U,V_{w}): w\in\Omega\}$ and  $\Gamma= \{\Gamma_{w}\in End_{\mathcal{A}}^{\ast}(U,V_{w}): w\in\Omega\}$ be two $(C-C^{'})$-controlled continue g-Bessel sequence for $U$ with respect to $\{V_{\omega}\}_{\omega \in \Omega}$ with bounds $E_{1}$ and $E_{2}$ respectively. Then the operator $L_{CC^{'}} : U \longrightarrow U$ given by:
\begin{equation}\label{2.13}
L_{CC^{'}}(x)=\int_{ \Omega}C^{'}\Gamma^{\ast}_{w}\Lambda_{\omega}Cxd\mu(w) \qquad \forall x \in U 
\end{equation}
is well defined and bounded with $\|L_{CC^{'}}\|\leq \sqrt{E_{1}E_{2}}$. Also its adjoint operator is $ L^{*}_{CC^{'}}(g)=\int_{ \Omega}C^{'}\Lambda^{\ast}_{w}\Gamma_{\omega}Cxd\mu(w)$
\end{proposition}
\begin{proof}
for any $x\in U$ and $\Psi \subset \Omega$, we have :
\begin{align*}
\|\int_{ \Psi}C^{'}\Gamma^{\ast}_{w}\Lambda_{\omega}Cxd\mu(w)\|^{2}&=\underset{y \in U, \|y\|=1}{\sup}\|\langle \int_{ \Psi} C^{'}\Gamma^{\ast}_{w}\Lambda_{\omega}Cxd\mu(w),y\rangle\|^{2}\\
&=\underset{y \in U, \|y\|=1}{\sup}\|\int_{ \Psi}\langle \Lambda_{\omega}Cx,\Gamma_{w}C^{'}y\rangle d\mu(w)\|^{2}\\
&\leq  \underset{y \in U, \|y\|=1}{\sup}\|\int_{ \Psi}\langle \Lambda_{\omega}Cx,\Lambda_{\omega}Cx\rangle d\mu(w)\|\|\int_{\Psi}\langle \Gamma_{w}C^{'}y,\Gamma_{w}C^{'}y\rangle d\mu(w)\|\\
&\leq \|\int_{ \Psi}\langle \Lambda_{\omega}Cx,\Lambda_{\omega}Cx\rangle d\mu(w)\|E_{2}\\
&\leq E_{1}E_{2}\|x\|^{2}
\end{align*}
since $\Psi $ is arbitrary,  $\int_{ \Psi}C^{'}\Gamma^{\ast}_{w}\Lambda_{\omega}Cxd\mu(w)$ converge in $U$ and 
\begin{equation*}
\|L_{CC^{'}}\|=\|\int_{ \Psi}C^{'}\Gamma^{\ast}_{w}\Lambda_{\omega}Cxd\mu(w)\|\leq \sqrt{E_{1}E_{2}}
\end{equation*}
In other hand, we have:
\begin{align*}
\langle L_{CC^{'}}x,y\rangle &= \langle \int_{ \Psi}C^{'}\Gamma^{\ast}_{w}\Lambda_{\omega}Cxd\mu(w),y\rangle\\
&=\int_{ \Psi}\langle C^{'}\Gamma^{\ast}_{w}\Lambda_{\omega}Cx,y\rangle d\mu(w)\\
&=\int_{ \Psi}\langle x,C\Lambda^{\ast}_{w}\Gamma_{\omega}C^{'}y\rangle d\mu(w)\\
&=\langle x,\int_{ \Psi}C\Lambda^{\ast}_{w}\Gamma_{\omega}C^{'}y d\mu(w)\rangle
\end{align*}
Thus $L^{*}_{CC^{'}}(g)=\int_{ \Omega}C^{'}\Lambda^{\ast}_{w}\Gamma_{\omega}Cxd\mu(w)$
\end{proof}

\begin{theorem}
Let  $\Lambda= \{\Lambda_{w}\in End_{\mathcal{A}}^{\ast}(U,V_{w}): w\in\Omega\}$ be a $(C-C^{'})$-controlled continue g-frames for $U$ with respect to $\{V_{\omega}\}_{\omega \in \Omega}$ and  $\Gamma= \{\Gamma_{w}\in End_{\mathcal{A}}^{\ast}(U,V_{w}): w\in\Omega\}$  be a $(C-C^{'})$-controlled continue g-Bessel sequence for $U$ with respect to $\{V_{\omega}\}_{\omega \in \Omega}$. Assume that $C$ and $C^{'}$ commute with each other and commute with $\Gamma^{\ast}_{w}\Gamma_{w}$. If the operator $L_{CC^{'}}$ defined in \eqref{2.13} is surjective then $\Gamma= \{\Gamma_{w}: w\in\Omega\}$ is also a $(C-C^{'})$-controlled continue g-frames for $U$ with respect to $\{V_{\omega}\}_{\omega \in \Omega}$
\end{theorem}

\begin{proof}
Let  $\Lambda= \{\Lambda_{w}\in End_{\mathcal{A}}^{\ast}(U,V_{w}): w\in\Omega\}$ be a $(C-C^{'})$-controlled continue g-frames for $U$ with respect to $\{V_{\omega}\}_{\omega \in \Omega}$.\\
by theorem \ref{sspp}, the operator $T_{CC^{'}}:l^{2}(\{V_{w}\}_{w\in\Omega}) \rightarrow U$ given by
\begin{equation*}
T_{CC^{'}}(\{y_{w}\}_{w\in\Omega})=\int_{w\in\Omega}(CC^{'})^{\frac {1}{2}}\Lambda^{\ast}_{\omega}y_{\omega}d\mu(w) \qquad  \forall \{y_{w}\}_{w\in\Omega} \in l^{2}(\{V_{w}\}_{w\in\Omega})
\end{equation*}
is well defined and bounded operator.\\
By \eqref{2..3} its adjoint operator  $T^{\ast}_{CC^{'}}: U\rightarrow l^{2}(\{V_{w}\}_{w\in\Omega})$ given by
\begin{equation}
T^{\ast}_{CC^{'}}(x)=\{\Lambda_{\omega}(C^{'}C)^{\frac{1}{2}}x\}_{\omega \in \Omega}
\end{equation}

Since  $\Gamma= \{\Gamma_{w}: w\in\Omega\}$ is also a $(C-C^{'})$-controlled continue g-Bessel sequence for $U$ with respect to $\{V_{\omega}\}_{\omega \in \Omega}$.\\
Again by theorem \ref{sspp}, the operator  $K_{CC^{'}}:l^{2}(\{V_{w}\}_{w\in\Omega}) \rightarrow U$ given by
\begin{equation*}
K_{CC^{'}}(\{y_{w}\}_{w\in\Omega})=\int_{w\in\Omega}(CC^{'})^{\frac {1}{2}}\Gamma^{\ast}_{\omega}y_{\omega}d\mu(w) \qquad \forall \{y_{w}\}_{w\in\Omega} \in l^{2}(\{V_{w}\}_{w\in\Omega})
\end{equation*}
is well defined and bounded operator. Again its adjoint operator is given by 
\begin{equation*}
K^{\ast}_{CC^{'}}(x)=\{\Gamma_{\omega}(C^{'}C)^{\frac{1}{2}}x\}_{\omega \in \Omega} \qquad \forall x \in U
\end{equation*}
Hence for any $x \in U$, the operator defined in \eqref{2.13} can be written as :
\begin{equation*}
L_{CC^{'}}(x)=\int_{ \Omega}C^{'}\Gamma^{\ast}_{w}\Lambda_{\omega}Cxd\mu(w)=K_{CC^{'}}T^{\ast}_{CC^{'}}x 
\end{equation*}
Since $L_{CC^{'}}$ is surjective then for any $x\in U$, there exists $y\in U$ such that:\\ $x=L_{CC^{'}}x=K_{CC^{'}}T^{\ast}_{CC^{'}}x$ and $T^{\ast}_{CC^{'}}x \in l^{2}(\{V_{w}\}_{w\in\Omega})$\\
This implies that $K_{CC^{'}}$ is surjective. As a result of lemma\ref{3}, we have $K^{\ast}_{CC^{'}}$ is bounded below, that is there exists $m>0$ such that:
\begin{equation*}
\langle K^{\ast}_{CC^{'}}x,K^{\ast}_{CC^{'}}x\rangle \geq m\langle x,x\rangle \qquad \forall x\in U
\end{equation*}
\begin{equation*}
\Longrightarrow \langle K_{CC^{'}}K^{\ast}_{CC^{'}}x,x\rangle \geq m\langle x,x\rangle \qquad \forall x\in U
\end{equation*}
\begin{equation*}
\Longrightarrow \langle\int_{ \Omega}(CC^{'})^{\frac {1}{2}}\Gamma^{\ast}_{w}\Gamma_{\omega}(CC^{'})^{\frac {1}{2}}xd\mu(\omega),x\rangle \geq m\langle x,x\rangle \qquad \forall x\in U
\end{equation*}
\begin{equation*}
\Longrightarrow \int_{ \Omega} \langle \Gamma_{w}Cx,\Gamma_{w}C^{'}x\rangle d\mu(\omega) \geq m\langle x,x\rangle \qquad \forall x\in U
\end{equation*}
Hence $\Gamma= \{\Gamma_{w}: w\in\Omega\}$ is also a $(C-C^{'})$-controlled continue g-frames for $U$ with respect to $\{V_{\omega}\}_{\omega \in \Omega}$
\end{proof}

\bibliographystyle{amsplain}

\end{document}